\documentclass[a4paper,11pt,leqno]{article}   %{amsart}  %  
\usepackage{amssymb, amsmath, amsthm, graphicx, color, epsfig, amsfonts, tikz, url}

\usepackage{thmtools} 
\newtheorem{thm}{Theorem}[section]
\newtheorem{lem}[thm]{Lemma}

\declaretheoremstyle[bodyfont=\normalfont,qed=$\square$]{remark}
\declaretheorem[style=remark,numberlike=thm,name=Remark]{rem}
\numberwithin{equation}{section}

\renewcommand{\a}{\alpha}
\renewcommand{\b}{\beta}
\newcommand{\e}{\varepsilon}
\newcommand{\de}{\delta}

\newcommand{\si}{\sigma}

\renewcommand{\t}{\tau}

\newcommand{\De}{\Delta}
\newcommand{\Ga}{\Gamma}

\def\R{{\mathbb{R}}}
\def\N{{\mathbb{N}}}
\def\Z{{\mathbb{Z}}}
\def\T{{\mathbb{T}}}

\allowdisplaybreaks

\newcommand{\ue}{u^\e}

\renewcommand{\div}{\operatorname{div}}

\renewcommand{\tilde}{\widetilde}
\renewcommand{\hat}{\widehat}

\title{Constant-speed interface flow from unbalanced Glauber-Kawasaki dynamics}
\author{Tadahisa Funaki$\,^{1)}$, Patrick van Meurs$\,^{2)}$, Sunder Sethuraman$\,^{3)}$ and Kenkichi Tsunoda$\,^{4)}$}
\begin{document}
\maketitle

\begin{abstract}
We derive the hydrodynamic limit of Glauber-Kawasaki dynamics.
The Kawasaki part is simple and describes independent movement of the particles
with hard core exclusive interactions. It is speeded up in a diffusive space-time scaling. The Glauber part describes the birth and death of particles. It is set to favor two levels of particle density with a preference for one of the two. It is also speeded up in time, but at a lesser rate than the Kawasaki part. Under this scaling, the limiting particle density instantly takes either of the two favored density values. The interface which separates these two values evolves with constant speed (Huygens' principle).

Similar hydrodynamic limits have been derived in four recent papers. The crucial difference with these papers is that we consider Glauber dynamics which has a preferences for one of the two favored density values. As a result, we observe limiting dynamics on a shorter time scale, and the evolution is different from the mean curvature flow obtained in the four previous papers. While several steps in our proof can be adopted from these papers, the proof for the propagation of the interface is new. 

\footnote{
\hskip -6mm 
${}^{1)}$ 
Beijing Institute of Mathematical Sciences and Applications, Huairou district,
Beijing, China.
e-mail: funaki@ms.u-tokyo.ac.jp \\
${}^{2)}$ Faculty of Mathematics and Physics, Kanazawa University, Kakuma, Kanazawa 920-1192, Japan.
e-mail: pjpvmeurs@staff.kanazawa-u.ac.jp \\
${}^{3)}$ Department of Mathematics, University of Arizona,
621 N.\ Santa Rita Ave., Tucson, AZ 85750, USA. e-mail: sethuram@math.arizona.edu\\
${}^{4)}$ Faculty of Mathematics, Kyushu University, 744 Motooka, Nishi-ku, Fukuoka, 819-0395, Japan. e-mail: tsunoda@math.kyushu-u.ac.jp}
\footnote{
\hskip -6mm
Keywords: Hydrodynamic limit, Huygens' principle, Glauber-Kawasaki dynamics, Sharp interface limit.}
\footnote{
\hskip -6mm
Abbreviated title $($running head$)$: Interface flow from Glauber-Kawasaki dynamics}
\footnote{
\hskip -6mm
2020MSC: 60K35, 82C22, 74A50.}
\end{abstract}

\paragraph{Acknowledgements}

TF was supported in part by JSPS KAKENHI Grant Number JP18H03672.
PvM was supported by JSPS KAKENHI Grant Number JP20K14358. 
SS was supported by grant ARO W911NF-181-0311.
KT was supported by JSPS KAKENHI Grant Number JP18K13426 and JP22K13929.

%\tableofcontents

\section{Introduction}

This paper fits to the recent paper series \cite{EFHPS,EFHPS-2,FvMST,FT} on hydrodynamic limits of interacting particle systems. In these papers the particle systems are such that the particle density on the macroscopic scale tends to be close to either of two distinct positive values $0 < \alpha_- < \alpha_+$. The interface which separates these two values turns out to move in time by the mean curvature flow. The contribution of the present paper to the series is that we consider a change in the setting of the particle system for which the hydrodynamic limit is \textit{different} from the mean curvature flow.

\subsection{The Glauber-Kawasaki dynamics}
\label{s:intro:GK}

For the interacting particle system we consider the Glauber-Kawasaki dynamics.
These dynamics are described by an exclusion process where the particles move on the $d$-dimensional discrete torus $\T_N^d = (\Z/N\Z)^d=\{1,2,\ldots,N\}^d$ of size $N$, where $d \geq 1$ is fixed. The Glauber part describes the birth and death of the particles. The Kawasaki part is simple and
describes the exchange of occupancies between neighboring sites.

In more detail, the particle configuration space is $\mathcal{X}_N = \{0,1\}^{\T_N^d}$ with elements denoted by $\eta=\{\eta_x\}_{x\in\T_N^d}$. We interpret $\eta_x = 1$ as the presence of a particle at site $x$ and $\eta_y = 0$ as that site $y$ is vacant. The generator of the Glauber-Kawasaki dynamics is given by 
\begin{equation} \label{LN}
  L_N = \frac{ N^2 }{\sqrt K} L_K+\sqrt K L_G,  
\end{equation}
for a given positive number $K>1$. The generators of the Kawasaki and Glauber part are given for any function $f:\mathcal{X}_N\to\R$ by
\begin{align} \label{LK}
L_Kf (\eta)   
&  = \frac12 \sum_{ \substack{ x, y\in\T_N^d \\ |x-y|=1 }}
\left\{  f\left(  \eta^{x,y}\right)  -f\left(\eta\right)  \right\},   \\\notag
L_Gf (\eta)  
&  =\sum_{x\in\T_N^d} c_x(\eta)
  \left\{  f\left(  \eta^x \right)  -f\left(  \eta\right)  \right\},
\end{align}
respectively, where $\eta^{x,y} \in \mathcal{X}_N$ is the  configuration $\eta$ after an exchange happens between $x$ and $y$, i.e.\ 
\[
  (\eta^{x,y})_z 
  = \left\{ \begin{array}{ll}
    \eta_y
    &\text{if } z = x \\
    \eta_x
    &\text{if } z = y \\
    \eta_z
    &\text{otherwise,}
  \end{array} \right.
\]
and $\eta^x  \in \mathcal{X}_N$ is the configuration after a flip (i.e.\ birth or death of a particle) happens at $x$, i.e.\ 
\[
  (\eta^x)_z
  = \left\{ \begin{array}{ll}
    1-\eta_x
    &\text{if } z = x \\
    \eta_z
    &\text{otherwise.}
  \end{array} \right.
\]
Finally, the symbol $c_x(\eta) \geq 0$ in $L_G$ is the flip rate, which may depend both on the site $x$ and on the configuration $\eta$.
In order to describe the assumptions on $c_x$, we set $\mathcal X$ as the unbounded configuration space $\{0,1\}^{\Z^d}$.
We introduce $\t_x: \mathcal{X} \to \mathcal{X}$ (or $\mathcal{X}_N \to \mathcal{X}_N$)
as the translation of $\eta$ by $x\in \Z^d$ (or $\in \T_N^d$) defined by
$(\t_x\eta)_z= \eta_{z+x}$ for $z\in \Z^d$ (or $\in \T_N^d$, in which case $x+z$ is taken modulo $N$). For a function $f=f(\eta)$ on $\mathcal{X}$ or $\mathcal{X}_N$, 
we denote $\t_xf(\eta) \equiv f(\t_x\eta)$.
We assume that $c_x(\eta) = \t_x c(\eta)$ for some non-negative 
local function $c = c_0$ on $\mathcal{X}$ (regarded as that on
$\mathcal{X}_N$ for $N$ large enough). Since $\eta_0$ takes values in $\{0,1\}$, $c(\eta)$ can be decomposed as
\begin{equation*}
c(\eta) = c^+(\eta)(1-\eta_0)+c^-(\eta) \eta_0
\end{equation*}
for some local functions $c^\pm$ which do not depend on $\eta_0$. The remaining assumptions on $c$ are given below by (BS) and (UB).

In addition to $c_x(\eta)$, one could consider an exchange rate $c_{x,y}(\eta)$ in the Kawasaki dynamics (called speed change); see e.g.\ \cite{FvMST}. In the present paper we simply take $c_{x,y}(\eta) \equiv 1$, i.e.\ occupancy exchanges between any two sites $x$ and $y$ are equally likely to happen, irrespective of the configuration $\eta$. We motivate this choice in Section \ref{s:intro:disc} below.

The scaling of the constants $N$ and $K$ in \eqref{LN} is such that we obtain nontrivial dynamics in the hydrodynamic limit. While we keep $K$ constant for now, we will consider $K=K(N)\to\infty$ as $N\to\infty$ later to observe a sharp interface in the limit. 
We write $\sqrt{K}$ rather than $K$ to be consistent with the setting in the
previous papers \cite{EFHPS,EFHPS-2,FvMST,FT}. 

To describe the remaining assumptions on the flip rate $c$, we keep $K\ge 1$ independent of $N$ for the moment. Then, as shown in \cite{DFL},
the empirical density (see \eqref{alN} below) of the process $\eta^N(t)$ generated by $L_N$ converges to the particle density profile $u^\e$, which is the solution
of the reaction-diffusion equation (of Allen-Cahn type) given by 
\begin{align} \tag{$P^\e$}
\label{Pe}
	\begin{cases}
	\partial_t u^\e
	= \e \Delta u^\e
	+ \displaystyle{ \frac{1}{\varepsilon}} f( u^\e)
	&\mbox{ in } (0,\infty)\times \T^d \\
    u^\e(0,v) = u_0(v)
	&\text{ for } v \in \T^d,
	\end{cases}	
\end{align}
where $\T^d$ is the $d$-dimensional continuous torus $(\R/\Z)^d=[0,1)^d$,
$u_0$ is an initial condition,
\[
  \varepsilon \equiv \frac1{\sqrt K} > 0
\] 
and $f$ is defined in terms of $c$ by 
\begin{align*}
  f(u) 
&\equiv E^{\nu_u}[(1-2\eta_0) c(\eta)]    \\\notag
&=  E^{\nu_u}[c^+(\eta) (1-\eta_0) -c^-(\eta) \eta_0] \\ \notag
& = (1-u) E^{\nu_u}[c^+(\eta)] - u E^{\nu_u}[c^-(\eta)].
\end{align*}
Here, $\nu_u$ is the Bernoulli measure on $\mathcal{X}$ with mean $u\in [0,1]$.
Note that, since $c$ is a local function, $f$ is a polynomial. We assume that $f$ satisfies the following properties (see Figure \ref{fig:f} for an example):  

\begin{itemize} 
  \item[(BS)] $f$ is bistable, i.e.\ $f \in C^2([0,1])$ has exactly three zeros $f(\alpha_-) = f(\alpha_+) = f(\alpha_*) = 0$ with $0<\alpha_- < \alpha_* < \alpha_+ < 1$, 
\begin{align*} %\label{cond_f_bistable}
	f'(\alpha_-) < 0, \quad f'(\alpha_+) < 0, \quad f'(\alpha_*) > 0.
\end{align*}
\item[(UB)] $f$ is \textit{unbalanced}, i.e.\ we assume
\begin{equation*}
  \int_{\a_-}^{\a_+}f(u)du \neq 0.
\end{equation*}
For convenience, we fix the sign as 
\begin{equation} \label{UB}
  \int_{\a_-}^{\a_+}f(u)du>0.
\end{equation}
\end{itemize}
This completes the list of assumptions on the flip rate $c$. The difference in these assumptions with respect to the previous papers \cite{EFHPS,EFHPS-2,FvMST,FT} is (UB); in those papers $f$ is instead balanced, i.e.\ the integral in (UB) is zero.
 Considering the unbalanced case has a significant effect on the macroscopic behavior of the Glauber-Kawasaki dynamics. Under (UB), the macroscopic dynamics are asymptotically faster than in the balanced case. Hence, in the present paper we slow down time to a shorter time scale in order to derive the macroscopic dynamics. In Section \ref{s:time:rescale} we describe in detail that our setting can be translated to the setting in the previous papers in terms of a rescaling of time.
 
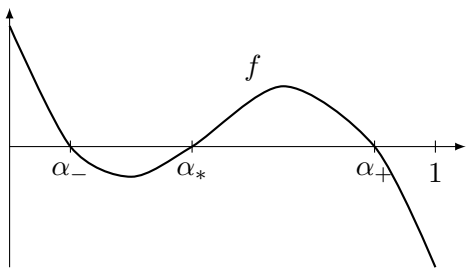
\begin{figure}[ht]
\centering
\begin{tikzpicture}[scale=.8, >= latex]
  \def \r {0.4}
  \def \lgray {black!20!white}
  
  \draw[->] (0,0) --++ (7.5,0);
  \draw[->] (0,-2) --++ (0,4.3);
    
  \draw [thick] plot [smooth] coordinates {(0,2) (1,0) (2,-.5) (3,0) (4.5,1) (6,0) (7, -2)};
  
  \draw (1,.1) --++ (0,-.2) node[below]{ $\alpha_-$};
  \draw (3,.1) --++ (0,-.2) node[below]{ $\alpha_*$};
  \draw (6,.1) --++ (0,-.2) node[below]{ $\alpha_+$};
  \draw (4,1.3) node{ $f$};
  \draw (7,.1) --++ (0,-.2) node[below]{$1$};
\end{tikzpicture} \\
\caption{An example of a function $f$ which satisfies all required properties.}
\label{fig:f}
\end{figure} 

While the assumptions on $c$ are similar in the previous papers \cite{EFHPS,EFHPS-2,FvMST,FT}, the assumptions on the exchange rates are different. \cite{FT} considers the same generator $L_K$ as in \eqref{LK}. \cite{FvMST} considers more general exchange rates $c_{x,y}(\eta)$ for which the diffusion term $\e \Delta u^\e$ in the corresponding problem \eqref{Pe} may result in the (typically) nonlinear diffusion term $\e \Delta P(u^\e)$ for some polynomial $P$. Finally, \cite{EFHPS,EFHPS-2} consider, instead of $L_K$, the generator of a zero-range process. In that case, the corresponding problem \eqref{Pe} has a similar nonlinear diffusion term as in \cite{FvMST}, and in addition (because the zero-range process is not an exclusion process) the density $u^\e$ need not be bounded from above by $1$. We comment below in Section \ref{s:intro:disc} on the result of these different choices for $L_K$ on the hydrodynamic limit.

\subsection{The limiting equation: Huygens' principle} 
\label{s:intro:P0}

We turn back to our point of interest, which is the hydrodynamic limit of the Glauber-Kawasaki process $\eta^N(t)$ for when $K = K(N) \to \infty$ as $N \to \infty$. It will turn out that this limiting equation is given by the limiting problem of \eqref{Pe} as $\e \to 0$. To describe it, let us examine \eqref{Pe} first. \eqref{Pe} is well-studied; see e.g.\ \cite{Fi,FH,F16,Ga}. 
It satisfies a comparison principle (see Lemma \ref{l:CP:Pe}) and possesses a unique, classical solution $\ue$. Moreover, the evolution of $u^\e$ is as follows. First, on a small time interval of size $O(\e)$, the reaction term $\e ^{-1}f(u^\e)$ dominates such that the PDE is well
approximated by the ordinary differential equation $u_t = \frac1\e f(u)$.
Due to the bistable nature of $f$, $\ue$ approaches
either of the values $\alpha_-$ or $\alpha_+$, and a diffusive interface is formed between the
regions $\{\ue\approx \alpha_-\}$ and $\{\ue\approx \alpha_+\}$. Once such an
interface is developed, the diffusion term $\e \Delta u$ becomes large
near the interface, and comes to balance with the reaction term so
that the interface starts to propagate beyond $t = O(\e)$ on an $O(1)$ time scale. 

In the limit $\e \to 0$, the diffusive interface of $u^\e$ converges to a sharp interface given by a hypersurface $\Gamma_t$ in $\T^d$, whose evolution is governed by the Huygens' principle
\begin{equation} \tag{$P^0$}
\label{P0}
 \begin{cases}
 \, V= c_*
 \quad \text { on } \Gamma_t \vspace{3pt}\\ 
 \, \Gamma_t\big|_{t=0}=\Gamma_0\,,
\end{cases}
\end{equation} 
where $c_* > 0$ is the speed of the traveling wave and $V$ is the normal velocity of $\Gamma_t$ from the $\alpha_+$-side to the 
$\alpha_-$-side; cf.\ \cite{BSS,Ga}. Figure \ref{fig:Gammat} illustrates the setting. The traveling wave speed $c_*$ is determined implicitly in terms of $f$ (or, equivalently, in terms of $c$); see \eqref{U} below and the discussion that follows. The sign of $c_*$ is determined by the assumption (UB); if the integral in \eqref{UB} would instead be negative, then $c_* < 0$ and the results in this paper apply with obvious modifications. 

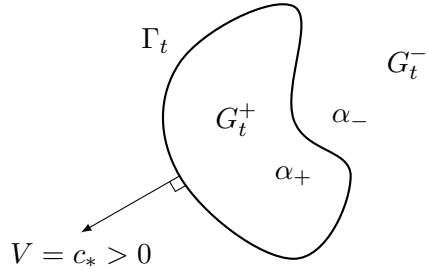
\begin{figure}[ht]
\centering
\begin{tikzpicture}[scale=1.5, >= latex]
\def \rr {0.04} 

%\draw[orange] (-2,-2) grid (2,2);

% the tangent is parallel to the line connecting the neighbouring two points
\draw [thick] plot [smooth cycle, tension=.7] coordinates {(0,0) (0,1) (1,1.5) (1,.5) (1.5, 0) (1, -0.732)};

\draw[->, rotate = 210] (0,0) -- (1,0) node[below] {$V = c_* > 0$};
\draw[rotate = 210] (0,.1) --++ (.1,0) --++ (0,-.1); 
\draw (0, 1) node[anchor = south east] {$\Gamma_t$};
\draw (.5,.5) node {$G_t^+$};
\draw (1,0) node {$\alpha_+$};
\draw (1.5,.5) node {$\alpha_-$};
\draw (2,1) node {$G_t^-$};
\end{tikzpicture} \\
\caption{A situation sketch of \eqref{P0}.}
\label{fig:Gammat}
\end{figure}

Also away from the interface, the solution $u_\e$ converges as $\e \to 0$. Its limit is given by the step function
\begin{equation*} %\label{chi} 
\chi_{\Gamma_t}(v) \equiv \begin{cases}
  \alpha_- &{\rm if \ } v \in \overline{G_t^-} \\
  \alpha_+ &{\rm if \ } v \in {G_t^+} 
\end{cases}
\end{equation*}
for all $v \in \T^d$, where $G_t^-$ and $G_t^+$ are the open regions of $\T^d$ separated by $\Gamma_t$; see Figure \ref{fig:Gammat}. To fix the sign, we take $G_t^-$ as the region corresponding to  $\{\ue\approx \alpha_-\}$ and $G_t^+$ as the region corresponding to $\{\ue\approx \alpha_+\}$. In terms of $G_t^+$, the solution to \eqref{P0} is given by
\begin{equation}  \label{eq:Gt}
\Ga_t= \partial G_t^+,
\qquad 
G_t^+ = \{v \in \T^d; \text{dist}\, (v, G_0^+)<c_* t\},
\end{equation}
where $\text{dist}\, (v,G_0^+)= \inf_{w\in G_0^+}|v-w|$.

\subsection{Main result: the hydrodynamic limit}

Our main result (Theorem \ref{t:main} below) on the hydrodynamic limit states that the process $\eta^N(t)$ converges to $\chi_{\Gamma_t}$ as $N \to \infty$ whenever $\eta^N(0)$ and $\chi_{\Gamma_0}$ are sufficiently close to each other. In order to turn this into a precise statement, we first introduce a proper scaling in which $\eta^N(t)$ can be compared with $\chi_{\Gamma_t}$. With this aim, we associate to any configuration $\eta \in \mathcal X_N$ the macroscopically scaled empirical measure $\alpha^N$ on $\T^d$ defined by 
\begin{align} \label{alN}
& \a^N(dv;\eta) = \frac1{N^d} \sum_{x\in \T_N^d}
\eta_x \de_{x/N}(dv),\quad v \in \T^d.
\end{align}
For the process $\eta^N(t)$, we set
\begin{align*}
& \a^N(t,dv) = \a^N(dv;\eta^N(t)), \quad t\ge 0.
\end{align*}

While both $\a^N(t, \cdot)$ and $\chi_{\Gamma_t}(\cdot)$ 
(more precisely, $\chi_{\Gamma_t}(\cdot) dv$)
are measures on $\T^d$, we need further preparation to define what it means for the initial conditions to be close. Let $\Gamma_0$ be given and let $\mu^N$ be an initial distribution for the process $\eta^N(\cdot)$. For technical reasons, we start from a regular alternative $u_0$ to the step function $\chi_{\Gamma_0}$. We interpret $u_0:\T^d \to (0,1)$ as the initial condition from \eqref{Pe}, and connect it to $\Gamma_0$ by requiring only that $\Gamma_0$ is the $\alpha_*$-level set of $u_0$, i.e.\  
\[
  \Ga_0 = \{v\in \T^d; u_0(v)=\a_*\}.
\]
Then, the regions $G_0^-$ and $G_0^+$ are determined by $\{ u_0 < \alpha_* \}$ and $\{ u_0 > \alpha_* \}$ respectively. We project $u_0$ onto the discrete torus by 
\begin{equation} \label{u0N}
  u_0^N(x) \equiv u_0(x/N)
  \qquad \text{for all } x \in \T_N^d.
\end{equation}
To $u_0^N$ we associate the (inhomogeneous) product measure 
\begin{equation} \label{nu0N}
  \nu_0^N \equiv \nu_{u_0^N(\cdot)} 
\end{equation}
on the configuration space $\mathcal X_N$, where 
\begin{equation} \label{nuu}
  \nu_{u(\cdot)}(\eta) \equiv \prod_{x \in \T_N^d} \nu_{u(x)} (\eta_x)
\end{equation}
is a product measure on $\mathcal X_N$ for any function $u : \T_N^d \to [0,1]$. 
Here, $\nu_{u(x)}$ is the measure on $\{0,1\}$ with mean $u(x)$. 
Now, note that $\nu_0^N$ and $\mu^N$ are both measures on $\mathcal X_N$. To measure how close they are, we use the relative entropy, which is defined for any two measures $\mu, \nu$ on $\mathcal X_N$ with $\nu$ having full support by
$$
H(\mu|\nu) \equiv \int_{\mathcal{X}_N} \frac{d\mu}{d\nu} \log \frac{d\mu}{d\nu} \, d\nu.
$$ 

With these connections between $\Gamma_0$, $u_0$, $u_0^N$ and $\nu_0^N$, we are ready to state our assumption on the initial conditions $\mu^N$ and $\Gamma_0$:
\begin{itemize}
\item [(BIP)] $H(\mu^N | \nu_0^N) = O(N^{d - \epsilon})$ as $N \to \infty$ for some $\epsilon > 0$ and for some $u_0 \in C^4(\T^d)$ which satisfies:
  \begin{itemize}
      \item $\Gamma_0 = \{ v \in \T^d : u_0(v) = \alpha_* \}$ is a closed $(d-1)$-dimensional hypersurface of class $C^4$ without boundary, where $\alpha_*$ is defined in (BS),
      \item $\displaystyle 0 < u_- \equiv \min_{v \in \T^d} u_0 (v) < \alpha_* < u_+ \equiv \max_{v \in \T^d} u_0 (v) < 1$, and
      \item $\nabla u_0 (v) \cdot {n}(v) \neq 0$ for any $v \in \Gamma_0$, 
      where ${n}(v)$ is the normal direction to $\Ga_0$ at $v$.
    \end{itemize}  
\end{itemize} 
Our motivation for (BIP) is as follows. On the one hand, $\mu^N$ has to be close to a product measure related to a regular density profile $u_0$. On the other hand, $u_0$ is relatively free to choose, in the sense that away from the level set $\Gamma_0 =  \{ u_0 = \alpha_* \}$, it can attain any value as long as it remains away from either $0$, $\alpha_*$ or $1$. This freedom in the choice of $u_0$ comes from the property of \eqref{Pe} that $u^\e$ will evolve towards $\chi_{\Gamma_0}$ in $O(\e)$ time.

Since under (BIP) $\Gamma_0$ is of class $C^4$, it follows from \cite{Fo} that there exists $T > 0$ such that $\Gamma_t$ is also of class $C^4$ for all $t \in [0,T]$ (see Section \ref{s:pf:PDE:prop} for details). We keep this time horizon $T$ fixed in the sequel. Finally, given the initial distribution $\mu^N$ for the process $\eta^N(\cdot)$, we denote by $\mathbb P^{\mu^N}$
the process measure with respect to $\eta^N(\cdot)$. 
%\comm{Here is the argument for "\cite{Fo} $\implies \Gamma_t \in C^4$". \cite{Fo} says that the regularity of the distance function $\overline d(0, \cdot)$ (see \eqref{od}) is $C^k$ close enough to $\Gamma_0$ whenever $\Gamma_0$ is $C^k$ (for any $k \geq 2$). In our case, $k=4$. To get regularity in time, we see from \eqref{eq:Gt} that $\Gamma_t$ is the $0$-level set of $\overline d(t,v) \equiv \overline d(0,v) - c_* t$. It is easy to see that this expression for $\overline d(t,v)$ indeed satisfies the equation \eqref{dt:PDE}. Moreover, from this expression and \cite{Fo}'s result we immediately see that $\overline d(t,v)$ is $C^3$ regular in $v$ and $C^\infty$ regular in $t$. Then, the $C^3$ regularity of $\Gamma_t$ from the $0$-level of $\overline d(t, \cdot)$ follows from the implicit function theorem. QED.}

\begin{thm} [Hydrodynamic limit] \label{t:main}
Assume the conditions {\rm (BS), (UB)} and {\rm (BIP)}. Then, there exists a constant $\delta > 0$ such that if $K = K(N) \to \infty$ as $N \to \infty$ with $K \leq \delta \sqrt{ \log N }$ and $K$ increasing in $N$, then
\begin{equation*}
   \lim_{N \to \infty} \mathbb P^{\mu^N} \big( \big| \langle \alpha^N(t), \phi \rangle - \langle \chi_{\Gamma_t}, \phi \rangle \big| > \gamma \big) = 0
 \end{equation*} 
 for all $t \in (0, T]$, all $\gamma > 0$ and all $\phi \in C^\infty(\T^d)$.
\end{thm}

\subsection{Structure of the proof of Theorem \ref{t:main}}
\label{s:intro:pf:overview}

The structure of the proof of Theorem \ref{t:main} is the same as in \cite{EFHPS}: we apply a two-level breakdown of Theorem \ref{t:main} into three theorems 
(Theorems \ref{t:Ent:est}, \ref{t:gen:PNK} and \ref{t:prop:uN})
which we will prove independently. 

In preparation for describing this breakdown, we introduce the usual central difference approximation of \eqref{Pe} given in terms of $u^N(t,\cdot) = \{u^N(t,x)\}_{x\in \T_N^d}$ by
\begin{align} \tag{$P_N^K$} \label{PNK} 
\left\{ \begin{aligned}
  \partial_t  u^N(t,x) 
  &= \frac1{\sqrt K} \De^N  u^N(t,x) + \sqrt K f( u^N(t,x))
  &&\text{for } 0 < t < \infty \text{ and } x \in \T_N^d \\
  u^N(0,x) 
  &= u_0^N (x)
  &&\text{for } x \in \T_N^d, 
\end{aligned} \right.
\end{align} 
where $u_0^N$ is defined in \eqref{u0N} in terms of $u_0$, and
\begin{align*} %\label{eq:DeNi}
\De^N u(x) \equiv N^2 \sum_{i=1}^d \left(u(x+e_i) + u(x-e_i) - 2u(x)\right),
\end{align*}
for any $u(\cdot) = \{u(x)\}_{x\in\T_N^d}$ is the discrete Laplacian. Here, $\{e_i\}_{i=1}^d$ are the standard unit basis vectors of $\Z^d$. Similar to $\nu_0^N$ (recall \eqref{nu0N},\eqref{nuu}) we set $\nu_t^N \equiv\nu_{u^N(t,\cdot)}$ for each $t > 0$. We will compare $\nu_t^N$ to the distribution $\mu_t^N$ of $\eta^N(t)$.

In addition to $u^N$, we use its extension $\hat u^N$ defined on the continuous torus. To define $\hat u^N$, we associate to any $x = (x_i)_{i=1}^d \in \T_N^d$ the box 
\[
  B \Big( \frac{x}N,\frac1N \Big) 
  \equiv \prod_{i=1}^d \Big[  \frac{x_i}N - \frac1{2N}, 
\frac{x_i}N + \frac1{2N} \Big)
\subset \T^d
\]
with center $\frac{x}N$ and side length $\frac1N$ (note that the family of boxes $\{ B(\frac{x}N,\frac1N) \}_{x \in \T_N^d}$ are a tessellation of $\T^d$). Then, we define the step function 
\begin{equation}
\label{uNhat} 
 \hat u^N(t,v) \equiv \sum_{x\in \T_N^d} u^N(t,x) 1_{B(\frac{x}N,\frac1N)}(v), \quad v \in \T^d.
\end{equation}

With this preparation we return to the two-level breakdown of Theorem \ref{t:main}. On the first level we use that Theorem \ref{t:main} is a corollary (see \cite[Section 2.3]{EFHPS} for a proof) of the following two theorems:

\begin{thm}[Entropy estimate] \label{t:Ent:est}
Given the assumptions in Theorem \ref{t:main}
$$
H(\mu_t^N|\nu_t^N) = o(N^d) \quad \text{as $N\to\infty$, uniformly for all } t\in [0,T].
$$
\end{thm}

\begin{thm}
\label{t:PDE}
Assume (BS), (UB) and  (BIP). Let $K = K(N) \to \infty$ as $N \to \infty$.
Then, there exists an exponent $\theta > 0$ and a propagation speed $c_* > 0$ such that if $K = o(N^\theta)$ as $N \to \infty$, then for the corresponding hypersurface $\Gamma_t$ given by \eqref{eq:Gt} we have that
$$\lim_{N\rightarrow\infty} \hat u^N(t,v) =\chi_{\Gamma_t}(v)$$
 for all $t\in (0,T]$ and all $v \in \T^d \setminus \Gamma_t$.
\end{thm}

We remark that the constants $\theta, c_*$ are determined by $f$; see \eqref{phi} for an explicit expression for $\theta$ and see \eqref{U} and the description below it for the definition of $c_*$.

On the second level, we apply a further breakdown of Theorem \ref{t:PDE}. The reason for this is that, similar to the solution $u^\e$ of \eqref{Pe} (recall the discussion at the start of Section \ref{s:intro:P0}), the evolution of $u^N(t,\cdot)$ can be split into two separate phases. In the first phase, the generation phase, $u^N(t,\cdot)$ evolves towards a discrete version of a diffusive interface on a small time interval of order $O(1/\sqrt K)$; see Theorem \ref{t:gen:PNK} below for a sufficient estimate. In the second phase, the propagation phase, the diffusive interface described by $u^N(t,\cdot)$ evolves similar to $\Gamma_t$; see Theorem \ref{t:prop:uN} below. We will show that Theorem \ref{t:PDE} is a corollary of Theorem \ref{t:prop:uN}.

To summarize the above, the main three ingredients of the proof of Theorem \ref{t:main} are the entropy estimate (Theorem \ref{t:Ent:est}), the generation of the interface in $u^N(t,\cdot)$ (Theorem \ref{t:gen:PNK}) and the propagation of the interface in $u^N(t,\cdot)$ (Theorem \ref{t:prop:uN}). Theorems \ref{t:Ent:est} and \ref{t:gen:PNK} follow from two counterpart theorems from \cite{EFHPS-2,FvMST} after applying a proper time rescaling; we demonstrate this in Section \ref{s:time:rescale}. Theorem \ref{t:prop:uN}, however, does not follow from the previous studies, which can be seen from the limiting equation in \eqref{P0}, which is different from the mean curvature flow of the previous studies. Therefore, we give a self-contained proof of Theorem \ref{t:prop:uN}; see Section \ref{s:pf:PDE}.

\subsection{Discussion}
\label{s:intro:disc}

We consider Theorem \ref{t:PDE} on the convergence of \eqref{PNK} to \eqref{P0} as interesting in its own right. For this reason we state Theorem \ref{t:PDE} for a larger range of values for $K$ than in our main Theorem \ref{t:main}. 

The convergence of \eqref{PNK} to \eqref{P0} as stated in Theorem \ref{t:PDE} is not a simple consequence of the convergence of \eqref{Pe} to \eqref{P0}, which is shown by \cite{BSS,Ga}. When working with the discrete counterpart \eqref{PNK} of \eqref{Pe}, more precise estimates on the solutions $u^N$ and $u^\e$ are needed. We refer to Section \ref{s:pf:PDE} for details.
\smallskip

Next we comment on the necessity of assumption (BIP) in Theorem \ref{t:main}. It is the intersection of the assumptions needed for the entropy estimate in Theorem \ref{t:Ent:est}, for the generation of the interface and for the propagation of the interface. More precisely, Theorem \ref{t:Ent:est} requires $u^N$ to have bounded discrete gradients up to 4th order. The generation of the interface \cite[Theorem 2.3]{EFHPS} requires $u^N(x) = u_0(x/N)$, and the propagation of the interface requires $u^0$ to be of class $C^3$.
\smallskip

In the next four comments we continue the discussion at the end of Section \ref{s:intro:GK} on the differences between the present study and \cite{EFHPS,EFHPS-2,FvMST,FT}. First, we argue that the hydrodynamic limit obtained in Theorem \ref{t:main} is consistent with the hydrodynamic limit obtained in \cite{EFHPS,EFHPS-2,FvMST,FT}, even though the limiting equation is different (in our paper this is the Huygens' principle \eqref{P0} whereas in the previous papers it is the mean curvature flow). The key difference in the setting is that our $f$ is unbalanced (see \eqref{UB}) whereas the $f$ used in the previous papers is balanced. In fact, Theorem \ref{t:main} also holds when $f$ is balanced; then, $c_* = 0$, i.e.\ the interface $\Gamma_t$ is stationary on this time scale. To observe a moving interface, one has to speed up the time variable from $t$ to $\sqrt K t$; see \cite{EFHPS}. Then, $\Gamma_t$ turns out to satisfy a mean curvature flow.

Second, the hydrodynamic limit in Theorem \ref{t:main} is also interesting for $d=1$. Indeed, for $d=1$, $\Gamma_t$ is a collection of points, and \eqref{P0} dictates that these points move with constant speed. In the previous studies where the limiting dynamics for $\Gamma_t$ is the motion by mean curvature, $\Gamma_t$ is stationary when $d=1$. This is consistent with mean curvature flow in the sense that points in one dimension have zero mean curvature.

Third, we comment on the choice of the Kawasaki generator $L_K$ in \eqref{LK}. We have chosen a simple version of $L_K$ such that the diffusion in \eqref{Pe} (and the corresponding discretization in \eqref{PNK}) are linear, as opposed to the nonlinear problems which appear in \cite{EFHPS,EFHPS-2,FvMST}. The reason for this is that in the nonlinear case the ODE in \eqref{U} below (in the proof for the propagation of the interface) gets additional nonlinear terms.  
 We leave this to future research.
 Nevertheless, Theorem \ref{t:main} also holds for the more general Kawasaki generator considered in \cite{FvMST} as long as the corresponding problem \eqref{Pe} has linear diffusion. Indeed, this is easy to see from the breakdown of the proof of Theorem \ref{t:main} into Theorems \ref{t:Ent:est} and \ref{t:PDE}, and that we use \cite{FvMST} for the proof of Theorem \ref{t:Ent:est}. 
 
Fourth, as an alternative to the Kawasaki generator $L_K$, one can consider the generator of a zero-range process (as done in \cite{EFHPS,EFHPS-2}) as long as the resulting diffusion term in \eqref{Pe} is linear. Then, a similar hydrodynamic limit as in Theorem \ref{t:main} holds. Indeed, also for such a zero-range process the problems \eqref{Pe} and \eqref{PNK} are similar. Then, the breakdown of the proof mentioned in Section \ref{s:intro:pf:overview} shows that it is left to establish an entropy estimate as in Theorem \ref{t:Ent:est}. This entropy estimate can be obtained from that in \cite[Theorem 2.2]{EFHPS} in a similar manner as our proof in Section \ref{s:time:rescale} of Theorem \ref{t:Ent:est}, which relies on the entropy estimate \cite[Theorem 1.3]{FvMST}.
\smallskip

Finally, we mention another recent result on a sharp interface limit of Glauber-Kawasaki dynamics. In \cite{DFPV} a two-species
Glauber-Kawasaki dynamics is studied. There, the Glauber part is only active at sites where two particles of different species meet. At such sites, it describes a high death rate for both particles. The resulting hydrodynamic limit describes the segregation of the two different species by means of a
Stefan free boundary problem.
\smallskip

The remainder of the paper is organized as follows. In Section \ref{s:time:rescale} we show the connection between $L_N$, \eqref{Pe} and \eqref{PNK} and their counterparts in \cite{EFHPS,EFHPS-2,FvMST,FT} through a time rescaling, and use it to prove Theorems \ref{t:Ent:est} and \ref{t:gen:PNK}. In Section \ref{s:pf:PDE} we prove Theorem \ref{t:main}, which essentially consists of the proof of Theorem \ref{t:PDE}.

\section{Entropy estimate and generation of interface by time rescaling}
\label{s:time:rescale}

We show that, by the time rescaling
\begin{equation*}
  \tilde t \equiv \e t = \frac1{\sqrt K} t,
\end{equation*}
Theorems \ref{t:Ent:est} and \ref{t:gen:PNK} follow essentially from \cite[Theorem 1.3]{FvMST} and \cite[Theorem 6.1]{EFHPS} respectively. For later use, we will also cite \cite[Theorem 1.1]{EFHPS-2}, which is the version of \cite[Theorem 6.1]{EFHPS} with $u^N$ replaced by $u^\e$. Care is needed when citing these theorems, because they are established for a nonlinear term $\tilde f$ which is balanced, i.e.
\[
  \int_{\a_-}^{\a_+} \tilde f(u) du = 0.
\]
However, the proofs of these theorems do not rely on this property, and extend without modification to the unbalanced $f$ considered in this paper.

Next, we introduce the time-rescaled versions of the process $\eta^N(\cdot)$ and of the problems \eqref{Pe} and \eqref{PNK}. Let 
\begin{equation*}
  \tilde \eta^N (\tilde t) \equiv \eta^N(\sqrt K \, \tilde t) = \eta^N(t)
\end{equation*}
be the speeded-up process with initial distribution given by 
$\tilde \mu^N=\mu^N$. Note that the generator of $\tilde \eta^N(\cdot)$ is given by
\begin{equation*} %\label{LNt}
  \tilde L_N = \sqrt K L_N = N^2 L_K + K L_G,  
\end{equation*}
which is (a simplification of) the generator considered in \cite{FvMST}.
Let $\tilde \mu^N_{\tilde t}$ be the distribution of $\tilde \eta^N (\tilde t)$, and note that 
\[
  \tilde \mu^N_{\tilde t} = \mu^N_t.
\]
The time-rescaled versions of \eqref{Pe} and \eqref{PNK} are given by
\begin{align} \tag{$\tilde P^\e$} %\label{Pet}
	 \left\{ \begin{aligned}
	\partial_{\tilde t} \tilde u^\e
	&= \Delta \tilde u^\e
	+ \displaystyle{ \frac{1}{\varepsilon^2}} f( \tilde u^\e)
	&&\mbox{ in } (0,\infty)\times \T^d \\
    \tilde u^\e(0,v) 
    &= u_0(v)
	&&\text{ for } v \in \T^d
	\end{aligned} \right.
\end{align}
and
\begin{align} \tag{$\tilde P_N^K$} %\label{PNKt}
 \left\{ \begin{aligned}
  \partial_{\tilde t} \tilde u^N
  &= \De^N \tilde u^N + K f(\tilde u^N)
  &&\text{in } (0,\infty) \times \T_N^d \\
  \tilde u^N(0,x) 
  &= u_0^N (x)
  &&\text{for } x \in \T_N^d,
\end{aligned} \right.
\end{align} 
respectively. The connection with the solutions $u^\e$ and $u^N$ to \eqref{Pe} and \eqref{PNK} is easily verified to be
\[
  \tilde u^\e(\tilde t,v) = u^\e(t,v),
  \qquad \tilde u^N(\tilde t, v) = u^N(t,v).
\]
As before, we set $\tilde\nu_{\tilde t}^N \equiv \nu_{\tilde u^N( \tilde t,\cdot)}$. Note that
\[
  \tilde\nu_{\tilde t}^N 
  = \nu_{\tilde u^N( \tilde t,\cdot)}
  = \nu_{u^N( t,\cdot)}
  = \nu_t^N.
\]
Hence, the properties on $\tilde\mu_{\tilde t}^N$, $\tilde\nu_{\tilde t}^N $, $\tilde u^\e$ and $\tilde u^N$ stated in \cite[Theorem 1.3]{FvMST}, \cite[Theorem 1.1]{EFHPS-2} and \cite[Theorem 6.1]{EFHPS} translate directly to properties on $\mu_t^N$, $\nu_t^N$, $u^\e$ and $u^N$. For the reader's convenience we translate these theorems to the original time $t$.

By rescaling time from $\tilde t$ to $t$, \cite[Theorem 1.3]{FvMST} captures Theorem \ref{t:Ent:est} when taking the following two notes into account:
\begin{enumerate}
  \item \cite[Theorem 1.3]{FvMST} only states the entropy estimate pointwise in $\tilde t$ whereas we require a uniform estimate in $\tilde t$ due to the $N$-dependent rescaling of time. However, the proof of \cite[Theorem 1.3]{FvMST} demonstrates that the entropy estimate is uniform in $\tilde t$ because the estimate is constructed by Gronwall's lemma.
  \item \cite[Theorem 1.3]{FvMST} requires the stronger upper bound on $K$ given by $K \leq \delta (\log N)^{\sigma / 2}$ for some implicit constant $\sigma \in (0,1)$ whereas we require $\sigma = 1$. The reason for this stronger upper bound is that \cite{FvMST} considers nonlinear diffusion of $u^N$ in \eqref{PNK}. In the case of nonlinear diffusion, it follows from \cite[Corollaries 4.4 and 5.10]{FS} 
  that there exists a constant $C > 0$ such that
\begin{equation} \label{uN:bds}
  N|u^N(t,y) - u^N(t,x)| \leq C K^{1/\sigma} 
  \quad \text{and} \quad 
  | \Delta^N u^N(t,x) | \leq C K^{2/\sigma} 
\end{equation}  
for all $t \in [0,T]$ and all $x,y \in \T_N^d$ with $|x-y|=1$. If these bounds would hold for $\sigma = 1$, then the proof of \cite[Theorem 1.3]{FvMST} shows that the weaker requirement $K \leq \delta \sqrt{ \log N }$ (as in the present paper) is sufficient. Since we consider linear diffusion, the bounds in \eqref{uN:bds} indeed hold with $\sigma = 1$; see \cite[Proposition 4.3]{FT} for the first of the two bounds, and then \cite[Proof of Proposition 6.3]{FS} for the second bound. 
\end{enumerate}

We continue with a simplified version of \cite[Theorem 1.1]{EFHPS-2} (see also \cite[Theorem 5.1]{EFHPS}) on the generation of the interface. This theorem formalizes that at a certain small time $t^\e$, $u^\e(t^\e, v)$ is close to either $\alpha_-$ or $\alpha_+$, except when $v$ is close to $\Gamma_0$. In preparation, we define
\begin{align}\label{cond_mu_eta0}
	\gamma \equiv f'(\alpha_*)
	, \quad
	t^\varepsilon \equiv \gamma^{-1} \varepsilon |\log \varepsilon|
	, \quad
	\delta_0 \equiv \min(\alpha_* - \alpha_-, \alpha_+ - \alpha_*).
\end{align}
The time point $t^\e$ is the analogue of $\tilde t^\varepsilon = \gamma^{-1}\varepsilon^2|\log \varepsilon| = \e t^\e$ in \cite{EFHPS,EFHPS-2} so that $u^\e(t^\e,v) = \tilde u^\e(\tilde t^\e,v)$ holds.
 
\begin{thm}\label{t:gen:Pe}
Let $f$ satisfy (BS). Let $u^\varepsilon$ be the solution to \eqref{Pe}. Then, for all $\delta \in (0, \delta_0)$ there exist positive constants $\varepsilon_0$ and $M_0$ such that, for all $\varepsilon \in (0, \varepsilon_0)$, we have the following: 
\begin{enumerate}
\item For all $v \in \T^d$,
\begin{align*} %\label{Thm_generation_i}
	\alpha_- - \delta
	\leq
	u^\varepsilon(t^\varepsilon,v)
	\leq
	\alpha_+ + \delta.
\end{align*}

\item If $u_0(v) \geq \alpha_* +  M_0 \varepsilon$, then
\begin{align*} %\label{Thm_generation_ii}
	u^\varepsilon(t^\varepsilon,v) \geq \alpha_+ - \delta.
\end{align*}

\item If $u_0(v) \leq \alpha_*  - M_0 \varepsilon$, then
\begin{align*} %\label{Thm_generation_iii}
	u^\varepsilon(t^\varepsilon,v) \leq \alpha_- + \delta.
\end{align*}

\end{enumerate}

\end{thm} 

Note that Theorem \ref{t:gen:Pe} provides bounds on $u^\e$ at the time point $t^\e$, but not at any other time $t \in (0, t^\e)$. For the purpose of proving Theorem \ref{t:PDE} it is not necessary to consider $u^\e$ on $(0, t^\e)$, because for any $t > 0$ fixed it holds that $t > t^\e$ for $\e$ small enough. Instead, use Theorem \ref{t:gen:Pe} in the proof of Theorem \ref{t:prop:uN} on the propagation of the interface.

Finally, we cite a simplified version of \cite[Theorem 6.1]{EFHPS}. It is the discrete counterpart of Theorem \ref{t:gen:Pe}. In preparation, we define 
\begin{align*} %\label{gam_bar:tN}
	\bar\gamma \equiv \max_{\zeta\in 
[u_- \wedge \a_-, u_+\vee\a_+]} -f'(\zeta) > 0, \quad
	t^N \equiv
\frac1{2 \gamma \sqrt K} \log K
\end{align*}
and
\begin{equation} \label{phi}
  \theta \equiv \frac {2\gamma}{3\gamma + \bar\gamma} \in (0, \tfrac23).
\end{equation}
The time $t^N$ is the analogue of $\tilde t^N = \frac1{2 \gamma K} \log K = K^{-1/2} t^N$ in \cite{EFHPS,EFHPS-2} so that $u^N(t^N,v) = \tilde u^N(\tilde t^N,v)$ holds. 

\begin{thm}  \label{t:gen:PNK}
Let $K \equiv K(N) \to \infty$ with $K = o(N^\theta)$ as $N \to \infty$.
Let $u^N(t,\cdot)$ be the solution to \eqref{PNK}. Then, for any $\de\in (0,\de_0)$ there exist $N_0, M_0>0$ such that the following hold for
every $N \ge N_0$: 

\noindent (1) For all $x\in \T_N^d$,
$$ 
\alpha_--\de\le u^N(t^N,x) \le \alpha_++\de.
$$
(2) If $u_0(\tfrac{x}N) \ge \a_*+M_0 K^{-1/2}$, then
$$
u^N(t^N,x) \ge \alpha_+-\de.
$$
(3) If $u_0(\tfrac{x}N) \le \a_*-M_0 K^{-1/2}$, then
$$
u^N(t^N,x) \le \alpha_-+\de.
$$
\end{thm}

\section{Proof of Theorem \ref{t:main}}
\label{s:pf:PDE}

From the overview of the proof of Theorem \ref{t:main} in Section \ref{s:intro:pf:overview} and the entropy estimate in Theorem \ref{t:Ent:est}, we infer that Theorem \ref{t:main} follows once Theorem \ref{t:PDE} is proven. In the remainder of this section we focus on proving Theorem \ref{t:PDE}. 

\subsection{Overview of the proof}
\label{s:pf:PDE:overview}

As preparation, we establish comparison theorems for \eqref{Pe} and \eqref{PNK} in Section \ref{s:pf:PDE:CP}.
Then, we prove the version of Theorem \ref{t:PDE} in which $\hat u^N$ is replaced by $u^\e$. In this proof we construct sub and super solutions $u_\e^\pm$ to \eqref{Pe} such that $u_\e^- \leq u^\e \leq u_\e^+$ and $u_\e^\pm (t, \cdot) \to \chi_{\Gamma_t}$ as $\e \to 0$. We do this in Section \ref{s:pf:PDE:prop}. Then, by using that \eqref{PNK} is the central differences approximation of \eqref{Pe}, we derive in Section \ref{s:pf:PDE:prop:uN} similar upper and lower bounds on $u^N$ as those on $u^\e$. In Section \ref{s:pf:PDE:pf} we put these bounds together to obtain Theorem \ref{t:PDE}.

\subsection{Comparison theorems for \eqref{Pe} and \eqref{PNK}}  
\label{s:pf:PDE:CP}

It is well-known that \eqref{Pe} satisfies a comparison principle. In order to state it, we recall that functions $u^+(t,v)$ and $u^-(t,v)$ are super and sub solutions to \eqref{Pe}
if $u^+$ and $u^-$ satisfy the PDE in \eqref{Pe} with \lq\lq$\ge$"
and \lq\lq$\le$" instead of \lq\lq$=$", respectively. Regarding notation, note that $u^\pm$ denote super/sub solutions and that $u_\pm$ are the given constants from (BIP).

\begin{lem}\label{l:CP:Pe}
Consider \eqref{Pe} with initial conditions $u^-(0,v) \le u^+(0,v)$.   Then, any corresponding super and sub solutions $u^+(t,v)$ and $u^-(t,v)$ to \eqref{Pe} satisfy
$$u^-(t,v) \le u^+(t,v) \qquad \text{for all } t\geq 0 \text{ and all } v \in \T^d.$$
Furthermore, suppose (BIP) holds for some $0<u_-<u_+<\infty$. Then,
for $t\ge 0$ and $v \in \T^d$, we have for the solution $u^\e$ to \eqref{Pe} with initial condition $u_0$ that
$$u_-\wedge \alpha_- \le u^\e(t,v)\le u_+\vee\alpha_+.$$
\end{lem}

The equation in \eqref{PNK}
 satisfies a comparison principle too; cf.\ \cite[Section 2.5]{FS}. 
We say that profiles $u(\cdot)=(u_x)_{x\in \T_N^d}$ and 
$v(\cdot) =(v_x)_{x\in \T_N^d}$ are ordered $u(\cdot)\ge v(\cdot)$ 
when $u_y\ge v_y$ for all $y\in \T_N^d$. We call $u^+(t,\cdot)$ and $u^-(t,\cdot)$ super and sub solutions to \eqref{PNK}
if $u^+$ and $u^-$ satisfy \eqref{PNK} with \lq\lq$\ge$"
and \lq\lq$\le$" instead of \lq\lq$=$", respectively.

\begin{lem}\label{l:CP:PNK}
Let $N \geq 1$ and consider \eqref{PNK} with initial conditions $u^-(0,\cdot) \le u^+(0,\cdot)$.   Then, any corresponding sub and super solutions $u^-(t,\cdot)$ and $u^+(t,\cdot)$ to \eqref{PNK} satisfy
$$u^-(t,\cdot) \le u^+(t,\cdot) \qquad \text{for all } t\geq 0.$$
Furthermore, suppose that (BIP) holds for the given $N$, i.e.\  $u_-\le u^N(0,\cdot)\le u_+$ for some $0<u_-<u_+<\infty$. Then,
for $t\ge 0$ and $x\in \T^d_N$, we have
$$u_-\wedge \alpha_- \le u^N(t,x)\le u_+\vee\alpha_+.$$
\end{lem}

\subsection{Propagation of the interface of $u^\e$}
\label{s:pf:PDE:prop}

Here we construct the sub and super solutions $u_\e^-$ and $u_\e^+$ and show that they bound $u^\e$ and $u^N$ from below and above, as mentioned in the overview Section \ref{s:pf:PDE:overview}. The constructions rely on the intuition that $u^\e$ resembles a smeared out version of $\chi_{\Gamma_t}$, denoted $w^\e$, which has the following structure. Let $v \in \Gamma_t$ and $n$ be the normal direction of $\Gamma_t$ at $v$ directed towards $G_t^-$. Then $w^\e (t, v + z n) = U(\frac z\e)$ for some $\e$-independent one-dimensional transition layer $U$, where $z \in (-\delta, \delta)$ for $\delta$ small enough with respect to $\Gamma_t$. 

\subsubsection*{Definition and properties of the transition layer $U$}

Let $\Gamma_t$ be as in \eqref{P0} for a certain $c_* > 0$ which we specify later. First, we show in more detail the existence of a certain $T > 0$ for which $\Gamma_t$ is of class $C^4$ for all $t \in [0,T]$. Let $\overline{d}(t,v)$ be the signed distance function to $\Gamma_t$ defined by
\begin{align} \label{od}
	\overline{d}(t,v)
	\equiv
	\begin{cases}
		{\rm dist}(v, \Gamma_t)
		& \text{ for } v \in \overline{G_t^-}
		\\
		- {\rm dist}(v, \Gamma_t)
		& \text{ for } v \in G_t^+.
	\end{cases}
\end{align}
The choice of sign is such that $\Gamma_t$ propagates along $\nabla \overline d$. Since $\Gamma_0$ is of class $C^4$, we have that $\overline d(0, \cdot)$ is of class $C^4$ in a neighborhood $\mathcal N_0$ of $\Gamma_0$; see e.g.\ \cite{Fo}. Recalling \eqref{eq:Gt}, there exists $T > 0$ such that $\overline d(t,v) \equiv \overline d(0,v) - c_* t$ for all $t \in [0,T]$ and all $v \in \mathcal N_0$. This has two consequences. First,
\begin{equation} \label{od:PDE}
  \partial_t \overline{d} = - c_* = - c_* | \nabla \overline{d} |
\end{equation}
on $(0,T) \times \mathcal N_0$. Note for the normal velocity $V$ of $\Gamma_t$ in \eqref{P0} that $V = -\partial_t \overline{d} = c_*$. 
Second, by the regularity of $\overline d(0, \cdot)$ we have that $\overline d(t,\cdot)$ is of class $C^4$ on $\mathcal N_0$ for any $t \in [0,T]$. Since $\Gamma_t$ is the $0$-level set of $\overline d(t,v)$, we obtain from the implicit function theorem that $\Gamma_t$ is of class $C^4$.

Next we derive an equation for $U$. Using $\overline d$ we extend the domain of definition of $w^\e$ as  
\[
  w^\e(t,v) \equiv U(\xi ), \quad \xi \equiv \frac{\overline d(t,v)}\e
\]
for all $t \in [0,T]$ and all $v \in \T^d$. By substituting $w^\e$ into the PDE in \eqref{Pe}, we obtain, for each of the three terms of the PDE, that
\begin{align*}
  \partial_t w^\e
  &= U' (\xi) \frac{\partial_t \overline d}\e, \\
  \e \Delta w^\e
  &= \div \big( U' (\xi) \nabla \overline d \big)
  = U'' (\xi) \frac{|\nabla \overline d|^2}\e + U' (\xi) \Delta \overline d
  , \\
  \frac1\e f(w^\e) 
  &= \frac1\e f(U(\xi)).
\end{align*}
By collecting the three terms of order $O(\e^{-1})$ and by assuming that $v$ is close enough to $\Gamma_0$ such that \eqref{od:PDE} holds, we obtain the following equation for $U$: 
\begin{align}\label{U}
	\begin{cases}
	U'' + c_* U' + f(U) 
	= 0 \qquad \text{on } \R,
	\\
	U(-\infty) = \alpha_+, \ U(0)= \alpha_*, \ U(\infty) = \alpha_-.
	\end{cases}
\end{align}
Here, the ``boundary conditions" $U(\pm \infty) = \alpha_{\mp}$ correspond to $u^\e$ being close to either $\alpha_+$ far enough inside $G_t^+$ or $\alpha_-$ far enough inside $G_t^-$. In addition, the condition $U(0)= \alpha_*$ fixes the horizontal shift of $U$, and matches with the condition (BIP) on $u_0$.

While our derivation of the problem in \eqref{U} is formal, the problem itself is stated rigorously. The following properties of \eqref{U} are well-known (see e.g.\ \cite[Proposition 2.1]{Ga} for a statement and \cite[p.101--108]{Fi} for a proof). Classical solutions $U$ to \eqref{U} exist only for a unique wave speed $c_*>0$. Hence, this value of $c_*$ is determined implicitly in terms of $f$. In the remainder, we keep $c_*$ fixed. Then, the classical solution $U$ to \eqref{U} is unique. Moreover, $U \in C^4(\R)$ and $U'<0$. %\comm{[About the regularity: $U$ classical solution $\implies U \in C^2$. Then, $f \in C^2 \implies f(U) \in C^2$. Then, $U'' = -c_* U' - f(U) \implies U \in C^4$]}

Next we list further properties of $U$ which we need later.

\begin{lem} \label{l:U:tails}
There exist constants $C, \lambda > 0$ such that $U$ satisfies
\begin{subequations} \label{U:tail:bds}
\begin{align}
  0 \leq \alpha_+ - U(z) & \leq C e^{\lambda z}, \\ %&&\text{for all } z < 0, \\
  0 \leq U(z) - \alpha_- & \leq C e^{-\lambda z}, \\ %&&\text{for all } z > 0, \\
  0 < -U'(z) & \leq C e^{-\lambda |z|} %&&\text{for all } z \in \R
\end{align}  
\end{subequations}
for all $z \in \R$. 
\end{lem}

\begin{proof}
Our proof is a modification of the proof of \cite[Lemma 2.1]{AHM}; see also [5, p. 101-108]. We write the second order ODE in \eqref{U} as the system of first order ODEs $w' = F(w)$, where $w  = (w_1, w_2)$ and $F(w) = (w_2, - c_* w_2 - f(w_1))$. Then, $w = (U, U')$ is a solution of this system, which connects the equilibrium $w_+ = (\alpha_+, 0)$ to  the equilibrium $w_- = (\alpha_-, 0)$. Linearizing the system around these equilibria, we find that the eigenvalues of the linear systems are given by
\[
  \lambda_p^\pm = - \frac{c_*}2 + p \sqrt{ \frac{c_*^2}4 - f'(\alpha_\pm) },
  \qquad \text{with } p = -1, +1.
\]
Since $f'(\alpha_\pm) < 0$, it follows that $\lambda_{-1}^\pm < 0 < \lambda_{+1}^\pm$. Hence, $(U(z), U'(z)) = w(z)$ is asymptotically given by $w_+ + r_+ e^{\lambda_{+1}^+ z}$ as $z \to -\infty$ and by $w_- + r_- e^{\lambda_{-1}^- z}$ as $z \to \infty$ for some given vectors $r_+, r_- \in \R^2$. Together with $U' < 0$ we conclude \eqref{U:tail:bds}.
\end{proof}

The following property is another estimate on $U'$, which we use in the proof of Lemma \ref{l:prop:ue}. In preparation for stating it, we define
\begin{equation} \label{beta}
  \beta \equiv \frac12 \min \{ |f'(\alpha_-)|, |f'(\alpha_+)| \}.
\end{equation}

\begin{lem} \label{l:Up:LB} 
For all $\sigma > 0$ small enough we have
\begin{equation} \label{Up:LB}
  U' \leq - \sigma (\beta + f'(U))
  \qquad \text{on } \R.
\end{equation}
\end{lem}

\begin{proof}
The proof is a modification of \cite[Lemma 8]{EFHPS-2}. First, we show that \eqref{Up:LB} holds on the interval $[R, \infty)$ for $R > 0$ large enough, uniformly in $\sigma$. For any $R > 0$ and any $z \geq R$, we have by Lemma \ref{l:U:tails} that
\begin{equation*}
   |f'(U(z)) - f'(\alpha_-)|
   \leq \|f''\|_{C([\alpha_-, \alpha_+])} |U(z) - \alpha_-|
   \leq C e^{-\lambda R}.
\end{equation*} 
for some constants $C, \lambda > 0$. Hence,
\begin{equation*}
  \beta + f'(U(z))
  \leq \frac12 |f'(\alpha_-)| +  f'(\alpha_-) + C e^{-\lambda R}
  = \frac12 f'(\alpha_-) + C e^{-\lambda R},
\end{equation*}
which is negative for $R$ large enough. Fixing such an $R$, the right-hand side of \eqref{Up:LB} is positive on $[R, \infty)$. Since $U' < 0$ on $\R$, it follows that \eqref{Up:LB} holds on $[R, \infty)$. By a similar argument, it follows that \eqref{Up:LB} also holds on $(-\infty, -R]$, possibly for a larger $R$.

It is left to show that \eqref{Up:LB} holds on $(-R, R)$ for $\sigma$ small enough with respect to $R$. Since $U' < 0$ on $\R$, we have that $U' \leq -c_R$ on $(-R,R)$ for some constant $c_R > 0$. Then, taking $\sigma$ small enough with respect to $R$, we obtain for any $|z| < R$ that
\begin{equation*}
  U'(z)
  \leq - c_R
  \leq - \sigma \big( \beta + \| f' \|_{C([\alpha_-, \alpha_+])} \big)
  \leq - \sigma (\beta + f'(U(z))).
\end{equation*}
This completes the proof of Lemma \ref{l:Up:LB}.
\end{proof}

\subsubsection*{Construction of the sub and super solutions $u_\e^-$ and $u_\e^+$}

With the intuition that $u^\e(t,v) \approx U( \overline d(t,v) / \e )$, we are ready to construct the sub and super solutions $u_\e^-$ and $u_\e^+$. The idea is to add three perturbations to $U( \overline d(t,v) / \e )$; a horizontal shift of $U$, a vertical shift of $U$ and a regularization of $\overline d$ (recall that $\overline d$ is regular only in a certain neighborhood of $\Gamma_t$). 

With this aim, we introduce a cut-off signed distance function $d=d(t,v)$. Let
\[
  Q_T \equiv (0,T) \times \T^d.
\] 
Take $d_0 > 0$ small enough so that the signed distance function $\overline{d}=\overline{d}(t,v)$ from the interface $\Gamma_t$ satisfies \eqref{od:PDE} on $\Omega_{3 d_0}$ and is of class $C^{4}$ in $\overline \Omega_{3 d_0}$, where the time-space tubular neighborhood $\Omega_\delta \subset Q_T$ of $\Gamma_t$ with thickness $\delta > 0$ is defined by
$$
\Omega_\delta \equiv \{
(t,v) \in Q_T \mid  | \overline{d}(t,v) | < \delta
\}. 
$$
Such a $d_0 > 0$ exists because $T$ is chosen such that $\Gamma_t$ is of class $C^4$ for all $t \in [0,T]$.
Let $h(s)$ be a smooth non-decreasing function on $\mathbb{R}$ such that $0 \leq h' \leq 1$ and 
$$
	h(s) = 
	\begin{cases}
		s & \text{if}~ |s| \leq d_0\\
		-2d_0 & \text{if}~ s \leq -3d_0\\
		2d_0 & \text{if}~ s \geq 3d_0.
	\end{cases}
$$
We then define the cut-off signed distance function $d$ by 
$$
	d(t,v) = h(\overline{d}(t,v)), ~~~ (t,v) \in \overline{Q_T}.
$$
Note that $d \in C^4 (\overline{Q_T})$ by construction. Moreover, since $d$ coincides with $\overline{d}$ in $\Omega_{d_0}$,
we have 
\begin{align} \label{dt:est}
  |\partial_t d|
  &= |h'(\overline d) \partial_t \overline d| \leq c_* 
  &&\text{on } \Omega_{3 d_0}, \\\label{dt:PDE}
  \partial_t d 
  &= h'(\overline d) \partial_t \overline d = -c_*
  &&\text{on } \Omega_{d_0}, \\\label{dt:0}
  \partial_t d 
  &= 0
  &&\text{on } Q_T \setminus \Omega_{3 d_0}
\end{align}
and, similarly,
\begin{align}  \label{nabd:est}
  |\nabla d|
  &= |h'(\overline d) \nabla \overline d| \leq 1 
  &&\text{on } \Omega_{3 d_0}, \\
%  \label{nabd:PDE}
%  \nabla d 
%  &= c_* |\nabla \overline d| = c_* 
%  &&\text{on } \Omega_{d_0}, \\
\label{nabd:0}
  \nabla d 
  &= 0
  &&\text{on } Q_T \setminus \Omega_{3 d_0}.
\end{align}

Using $d$, we construct the sub and super solutions as follows: given $0<\varepsilon <1$, we define  
\begin{equation}  \label{upm} 
	u^\pm(t,v) 
	\equiv u_\e^\pm(t,v)
	\equiv U
	\left(
		\frac{d(t,v)}\e \mp p(t)
	\right)
	\pm q(t), 
\end{equation}
where for the horizontal shift $p(t)$ and vertical shift $q(t)$ it turns out \textit{a posteriori} (for the purpose of proving Lemmas \ref{l:upm:0} and \ref{l:prop:ue} below) that
\begin{gather*}
    q(t) = 2 \sigma \beta \exp \Big(- \frac{ \beta t }{ 2 \varepsilon } \Big) + \frac{3 \e}\beta, \\
\begin{aligned}
    p(t) &=  L + C_{\Delta d} t + \frac1{\sigma \e} \int_0^t q(s) \, ds \\
    &= L + \Big( \frac3{\sigma \beta} + C_{\Delta d} \Big) t + 4 \Big( 1 - \exp \Big(- \frac{ \beta t }{ 2 \varepsilon } \Big) \Big)
\end{aligned}   
\end{gather*}
are appropriate choices,
where 
\[
  C_{\Delta d} \equiv \| \Delta d \|_{L^\infty (Q_T)}.
\]
Here, $\beta$ is defined in \eqref{beta} and $\sigma, L >0$ are constants which we specify below in Lemmas \ref{l:upm:0} and
\ref{l:prop:ue}. Note that $p, q \geq 0$. In comparison to the balanced case in \cite{EFHPS,EFHPS-2}, there are three differences:
\begin{enumerate}
  \item the expressions for $p,q$ are considerably different,
  \item in \eqref{upm} we write ``$\mp p(t)$" instead of ``$\pm p(t)$" because in the present paper we have switched the boundary conditions of $U$ at $\pm \infty$,
  \item we do not need a next order approximation term $\e U_1$ (also called corrector) of $U$ in \eqref{upm}.
\end{enumerate} 
Although we work on $\T^d$, if we take the viewpoint of working on $\R^d$, we may regard the signed distance function $d(t, \cdot)$ as periodic with period $1$ so that $u^\pm(t,\cdot)$ are periodic as well for all $t \in [0,T]$.

First we show that at $t = 0$, there exists an $L > 0$ such that $u^\pm(0,\cdot)$ sandwich both $u^\e(t^\e, \cdot)$ and $u^N(t^N, \cdot)$ for all $\sigma, \e, \frac1N$ small enough. 

\begin{lem}\label{l:upm:0}
There exist $L, \sigma_0, \e_0 > 0$
such that
\begin{align} \label{l:upm:0:ue}
	&u^-(0,v) 
	\leq 
	u^\e (t^\e,v)
	\leq 
	u^+(0,v)
\end{align}
for all $v \in \T^d$, all $\sigma \in (0,\sigma_0)$ and all $\varepsilon \in (0, \varepsilon_0)$.
Moreover, for $K = \frac1{\e^2} \to \infty$ with $K = o(N^\theta)$ (recall \eqref{phi}) as $N \to \infty$, there exists $N_0 \in \N$ such that 
\[
u^- \Big( 0, \frac xN \Big) 
\le u^N ( t^N, x ) 
\le u^+ \Big( 0, \frac xN \Big)
\]
for all $N \geq N_0$, all $\sigma \in (0,\sigma_0)$ and all $x \in \T_N^d$.
\end{lem}

\begin{proof}
The following proof only uses the bounds on $u^\e(t^\e, \cdot)$ and $u^N(t^N, \cdot)$ which are stated in Theorems \ref{t:gen:Pe} and \ref{t:gen:PNK}. Since these bounds are similar, it suffices to prove \eqref{l:upm:0:ue}. In addition, since the proof of the lower bound is similar to that of the upper bound, we focus on the latter.

Take $\sigma_0 \leq \delta_0 / (2\beta)$ (recall $\delta_0$ from \eqref{cond_mu_eta0}) small enough such that Lemma \ref{l:Up:LB} applies for any $\sigma \in (0, \sigma_0)$.  
Taking $\sigma \in (0, \sigma_0)$, we apply Theorem \ref{t:gen:Pe} with $\delta = \sigma \beta \in (0, \delta_0)$ to obtain constants $M_0, \e_0 > 0$ such that 
\[
  u^\e (t^\e,v)
  \leq \begin{cases}
    \alpha_- + \sigma \beta
    &\text{if } u_0(v) \leq \alpha_* - M_0 \e \\
    \alpha_+ + \sigma \beta
    &\text{otherwise} 
  \end{cases} 
\] 
for all $v \in \T^d$ and $\e \in (0,\e_0)$.
Then, by the non-degeneracy of $\nabla u_0$ on $\Gamma_0$ (see assumption (BIP)), there exists a constant $M_1>0$ such that
\[
  u^\e (t^\e,v)
  \leq \begin{cases}
    \alpha_- + \sigma \beta
    &\text{if } \overline d(0,v) > M_1 \e \\
    \alpha_+ + \sigma \beta
    &\text{otherwise} 
  \end{cases}
\]
for all $v \in \T^d$. By taking $\e_0$ smaller if necessary, we have that the condition $\overline d(0,v) > M_1 \e$ is equivalent to $d(0,v) > M_1 \e$.
Furthermore, let $L > M_1$ be large enough such that $U(M_1 - L) \geq \alpha_+ - \sigma \beta$. Then, since $U$ is decreasing and $U \geq \alpha_-$ on $\R$, we obtain from $p(0) = L$ and $q(0) \geq 2 \sigma \beta$ that
\[
  u^\e (t^\e,v) 
  \leq U ( M_1 - L ) + 2 \sigma \beta 
  \leq U \left( \frac{d(0,v)}\e - p(0) \right) + q(0)
  = u^+(0,v)
  %\qquad \text{for all } v \in \R. 
\]
for all $v \in \T^d$. This completes the proof.
\end{proof}

We proceed with showing sufficient properties of $u^\pm(t, \cdot)$ for all $t \in [0, T]$. With this aim, we rewrite the PDE in \eqref{Pe} as
\begin{equation} \label{cL}
  \mathcal{L}^\e u \equiv \partial_tu - \e \Delta u - \frac{1}{\varepsilon} f(u) = 0.
\end{equation}

\begin{lem}\label{l:prop:ue}
There exist $\sigma, L, \e_0 > 0$
such that
\begin{align} \label{Lem_Prop_subsuper:t0}
	&u^-(t - t^\e,v) 
	\leq 
	u^\e (t,v)
	\leq 
	u^+(t - t^\e,v)
	&& \text{for all } v \in \T^d, t \in [t^\e, T] \text{ and } \varepsilon \in (0, \varepsilon_0), \\\label{Lem_Prop_subsuper:vs}
    &\mathcal{L}^\e u^- \leq -1 < 1 \leq \mathcal{L}^\e u^+
	&&	\text{in } [0,T]\times\T^d \text{ for all } \varepsilon \in (0, \varepsilon_0) \text{, and} \\\label{Lem_Prop_subsuper:lim}
	&\lim_{\e \to 0} u^\pm(t - t^\e,v) = \chi_{\Gamma_t}(v)
	&& \text{for all } t \in (0,T] \text{ and } v \in \T^d \setminus \Gamma_t.
\end{align}
In particular, $u^\pm(t - t^\e,v)$ are sub and super solutions for Problem $(P^\varepsilon)$ for $t \in [t^\e,T]$.
\end{lem}

\begin{proof}
The proof is a modification of the arguments used in \cite[Sections 4.2--4.4]{EFHPS-2}. We only prove the properties related to $u^+$; those for $u^-$ can be proven similarly. 

Recall $\beta$ from \eqref{beta}, let $L, \sigma_0 > 0$ be as in Lemma \ref{l:upm:0} and take 
\begin{equation} \label{pfzy}
  \sigma < \min \left\{ \sigma_0, \frac1{4 \| f'' \|_{C([0, 2 \alpha_+])}} \right\}.
\end{equation}
First we prove \eqref{Lem_Prop_subsuper:lim}. Note that 
\[
  u^+(t - t^\e,v) 
	= U
	\left(
		\frac{d(t - t^\e,v)}\e - p(t - t^\e)
	\right)
	+ q(t - t^\e),
\] 
For the second term in the right-hand side, we obtain from $t - t^\e \to t > 0$ as $\e \to 0$ that $q(t - t^\e) \to 0$ as $\e \to 0$. For the first term, we observe from $L \leq p(t - t^\e) \leq C$ 
and $d(t - t^\e,v) \to d(t,v) \neq 0$ as $\e \to 0$ that 
\[
  U
	\left(
		\frac{d(t - t^\e,v)}\e - p(t - t^\e)
	\right) \to \chi_{\Gamma_t}(v)
\]
as $\e \to 0$. Then, \eqref{Lem_Prop_subsuper:lim} follows. 

Second, we prove \eqref{Lem_Prop_subsuper:vs}. In view of \eqref{cL} we set $\xi \equiv \frac d\e - p$ and compute and expand
\begin{align*}
  \partial_t u^+ 
  &= U' (\xi) \left( \frac{\partial_t d}\e - p' \right) + q', \\
  \e \Delta u^+
  &= \div \big( U' (\xi) \nabla d \big)
  = U'' (\xi) \frac{|\nabla d|^2}\e + U' (\xi) \Delta d
  , \\
  \frac1\e f(u^+) 
  &= \frac1\e f(U(\xi))  + \frac q\e f'(U(\xi)) + \frac{q^2}{2 \e} f''(U(\xi) + h q)
\end{align*}
for some $h \in (0,1)$ given by Taylor's Theorem applied to $f$ at $U(\xi)$. Substituting $U'' = -c_* U' - f(U)$, we obtain from \eqref{cL} that
\begin{align*}
  \mathcal L^\e u^+ 
  &= \frac1\e \Big( U' (\xi) \partial_t d + \big( c_* U'(\xi) + f(U(\xi)) \big) |\nabla d|^2 - f(U(\xi)) \Big) \\
  &\quad + \Big( -U' (\xi) (\Delta d + p') + q' - \frac q\e f'(U(\xi)) - \frac{q^2}{2 \e} f''(U(\xi) + h q) \Big) \\
  &\equiv \frac1\e E_{-1} + E_0. 
\end{align*} 

First we show that $E_{-1} \geq - \e$ for $\e$ small enough. On $\Omega_{d_0}$ it holds that $d = \overline d$, and thus $|\nabla d| = 1$. Then, it follows from \eqref{dt:PDE} that $E_{-1} = 0$ on $\Omega_{d_0}$. On $Q_T \setminus \Omega_{d_0}$ we have by \eqref{dt:est}--\eqref{nabd:0} that $|\nabla d| \leq 1$, $|\partial_t d| \leq c_*$ and $|\xi| \geq \frac{d_0}{2\e}$ for $\e$ small enough. Then, estimating 
\[
  |f(U(\xi))| 
  \leq |f(\alpha_\pm)| + |U(\xi) - \alpha_\pm| \| f' \|_{C([\alpha_-, \alpha_+])}
  = |U(\xi) - \alpha_\pm| \| f' \|_{C([\alpha_-, \alpha_+])}
\]
and using the bounds on $U, U'$ in Lemma \ref{l:U:tails}, we obtain by estimating all four terms of $E_{-1}$ independently that
\[
  \frac1\e E_{-1} \geq -\frac C\e e^{-\lambda \tfrac{d_0}{2\e}} \geq -1
  \qquad \text{on } Q_T \setminus \Omega_{d_0}
\]
for some constants $\lambda, C > 0$ and for $\e$ small enough. 

It is left to show that $E_0 \geq 2$. By the definition of $C_{\Delta d}$, we obtain
\[
  \Delta d + p' = \Delta d + C_{\Delta d} + \frac q{\sigma \e} \geq \frac q{\sigma \e},  
\]
which is positive. Together with Lemma \ref{l:Up:LB}, we obtain for the first term in $E_0$
\[
  -U' (\xi) (\Delta d + p') 
  \geq \big( \beta + f'(U(\xi)) \big) \frac q\e
  = 3 + \frac{2 \sigma \beta^2}\e e^{- \tfrac{\beta}{2 \e} t} + f'(U(\xi)) \frac q\e.
\]
Note that the third term in the right-hand side cancels with the third term of $E_0$. The second term of $E_0$ is given by
\[
  q' = - \frac{\sigma \beta^2}\e e^{- \tfrac{\beta}{2 \e} t},
\]
and the fourth term is estimated from below by (using \eqref{pfzy} and taking $\e$ small enough)
\[
  -\frac{q^2}{2 \e} f''(U(\xi) + h q) 
  \geq - \frac{ \| f'' \|_{C([0, 2 \alpha_+])} } \e \left( 4 \sigma^2 \beta^2 e^{- \tfrac{ \beta }{ \varepsilon } t} + \frac{9 \e^2}{\beta^2} \right)
  \geq - \frac{\sigma \beta^2}\e e^{- \tfrac{ \beta }{ \varepsilon } t} - 1.
\]
By substituting these estimates in the expression of $E_0$, we obtain $E_0 \geq 2$. This completes the proof of \eqref{Lem_Prop_subsuper:vs}.

Now, \eqref{Lem_Prop_subsuper:t0} is obvious: \eqref{Lem_Prop_subsuper:vs} implies that $u^\pm(t - t^\e,v)$ are sub and super solutions for Problem $(P^\varepsilon)$ for $t \in [t^\e,T]$, \eqref{l:upm:0:ue} demonstrates that \eqref{Lem_Prop_subsuper:t0} holds at $t = t^\e$, and thus \eqref{Lem_Prop_subsuper:t0} follows from the comparison principle Lemma \ref{l:CP:Pe}.
\end{proof}

\begin{rem}  \label{rem:2.2}\rm
In \eqref{Lem_Prop_subsuper:vs} the constant $1$ has no particular meaning; any positive constant can be obtained by choosing $p,q$ properly.
\end{rem}

\subsection{Propagation of the interface of $u^N$}
\label{s:pf:PDE:prop:uN}

We rely on Lemma \ref{l:prop:ue} to prove the discrete counterpart of \eqref{Lem_Prop_subsuper:t0}. We use $\sqrt K = \frac1\e$ and recall $\theta$ from \eqref{phi}.

\begin{thm}  \label{t:prop:uN}
Let $\b, \si, L >0$ be as in Lemma \ref{l:prop:ue}, and let $K=K(N) \to \infty$ with $K =o(N^\theta)$ as $N \to \infty$. Then,
there exists $N_0\in \N$ such that
\begin{equation*} %\label{eq:3.3-com}
u^- \Big( t - t^N, \frac xN \Big) 
\le u^N ( t,  x ) 
\le u^+ \Big( t - t^N, \frac xN \Big)
\end{equation*} 
holds for all $N\ge N_0$, all $t\in [t^N, T]$ and all $x \in \T_N^d$.
\end{thm} 

\begin{proof}
The upper and lower bound can be proven similarly; we focus on the upper bound. By the bounds at initial time given in Lemma \ref{l:upm:0} and the comparison principle Lemma \ref{l:CP:PNK}, it is sufficient to show that $u^+$ projected onto the discrete torus $\T_N^d$, i.e.\ $\overline u^+(t,x) \equiv u^+(t, \frac xN)$, is a super solution to \eqref{PNK}, i.e.
\begin{align}  \label{eq:7.3.1}
\mathcal{L}^N \overline u^+ 
\equiv \partial_t \overline u^+ - \frac1{\sqrt K} \De^N \overline u^+ - \sqrt K f(\overline u^+) \ge 0 \quad \text{on } (0,T) \times \T_N^d
\end{align}
for all $N$ large enough.  

To prove \eqref{eq:7.3.1}, we set $\e = \frac1{\sqrt K}$ and decompose
\begin{align}  \label{eq:7.3.2}
 \mathcal{L}^N \overline u^+ (t,x) 
 = \mathcal{L}^\e u^+ \Big( t, \frac xN \Big)
+ \frac1{\sqrt K} \Big( \De u^+ \Big( t, \frac xN \Big) - \De^N \overline u^+ (t,x) \Big),
\end{align}
where 
$\mathcal{L}^\e$ is as in \eqref{cL}.  By Lemma \ref{l:prop:ue}, $\mathcal{L}^\e u^+ (t, \frac xN) \geq 1$. For the second term in \eqref{eq:7.3.2}, since $u^+(t, \cdot) \in C^4 (\T^d)$ and since $\De^N \overline u^+ (t,x)$ is the central difference approximation of $\De u^+ ( t, \frac xN )$, 
we have (recalling that $L \leq p(t) \leq C$) 
\begin{align*}
\frac1{\sqrt K} \Big| \De u^+ \Big( t, \frac xN \Big) - \De^N \overline u^+ (t,x) \Big|
\leq \frac C{\sqrt K N} \| u^+(t, \cdot) \|_{C^3(\T^d)}
\leq C' \frac{K}{N}.  
\end{align*}
Since $K = o(N^\theta) = o(N^{2/3})$, this contribution to \eqref{eq:7.3.2} vanishes as $N \to \infty$. This completes the proof.
\end{proof}

\subsection{Proof of Theorem \ref{t:PDE}} 
\label{s:pf:PDE:pf} 

Let $t \in (0,T]$ and  $v\in\T^d\setminus \Gamma_t$ be arbitrary. Let $N$ be large enough such that $t > t^N$. Let $x^N\in \T_N^d$ be such that $v \in B(\frac{x^N}N, \frac1{2N})$ for all $N \in \N$, and note from \eqref{uNhat} that $\hat u^N(t,v) = u^N(t,x^N)$ and that $| v - \frac{x^N}N | \leq \frac CN$ for all $N \in \N$, where $C$ depends on the dimension $d$. Let $u^\pm$ be given by \eqref{upm} with constants $\b, \si, L >0$ as in \eqref{beta} and Lemma \ref{l:prop:ue}. By Theorem \ref{t:prop:uN} we have 
\begin{align} \label{pfzx}
  u^- \Big( t - t^N, \frac{x^N}N \Big) 
\le u^N ( t,  x^N ) 
\le u^+ \Big( t - t^N, \frac{x^N}N \Big)
\end{align}
for $N$ large enough. For the left- and right-hand side we note that 
\begin{align*}
  \Big| u^\pm \Big( t - t^N, \frac{x^N}N \Big) - u^\pm(t - t^N,v) \Big|
  &\leq \Big| \frac{x^N}N - v \Big| \| \nabla u^\pm(t - t^N,\cdot) \|_{C(\T^d)} \\
  &\leq \frac CN \| U' \|_{C(\R)} \sqrt K \| \nabla d \|_{C(Q_T)}
  \leq C' \frac{\sqrt K}N
\end{align*}
vanishes as $N \to \infty$. Hence, Theorem \ref{t:PDE} follows from \eqref{Lem_Prop_subsuper:lim} by passing to the limit $N \to \infty$ in \eqref{pfzx}.
\hfill \qed

\end{document}